\documentclass[11pt, a4paper]{amsart}
\usepackage{verbatim, latexsym, amssymb, amsmath}
\usepackage{epsfig}
\def\R{\mathbb R}
\def\N{\mathbb N}
\def\Z{\mathbb Z}
\def\C{\mathbb C}

\def\H{\mathcal H}

\def\x{\mathbf x}

\def\vol{\mathrm{vol}}
\def\re{\mathrm{Re}}

\def\a{\,d\mathcal{H}^n}

\newtheorem*{conj}{Conjecture}

\newtheorem{quest}{Question}[section]
\newtheorem{thm}{Theorem}[section]
\newtheorem{lemm}[thm]{Lemma}
\newtheorem{cor}[thm]{Corollary}
\newtheorem{prop}[thm]{Proposition}
\theoremstyle{remark}
\newtheorem{rmk}[thm]{Remark}

\theoremstyle{definition}

\title{Recent Progress on Singularities of Lagrangian Mean Curvature Flow}

\author{{Andr\'e Neves}}
\email{aneves@imperial.ac.uk}
\address{Department of Mathematics, Imperial College London, London SW7 2AZ}

\begin{document}

\maketitle 
\begin{center}
{\em Dedicated to Professor Richard Schoen on his sixtieth birthday.}
\end{center}
\markboth{Recent Progress on Singularities of LMCF} { Andr\'e Neves} \maketitle
\begin{abstract}
We survey some of the state of the art regarding singularities in Lagrangian mean curvature flow. Some open problems are suggested at the end.
\end{abstract}

\tableofcontents

\section{Introduction} Since Yau's solution to the Calabi Conjecture, Calabi-Yau manifolds and minimal Lagrangians (called special Lagrangians) have acquired  a central role in Geometry and Mirror Symmetry over the last 30 years. Unfortunately, the most basic question one can ask about special Lagrangians, whether they exist in a given homology or Hamiltonian isotopy class, is still largely open.  Special Lagrangians are  area-minimizing  and so one could  approach the existence problem by trying to minimize area among all Lagrangians in a given class. Schoen--Wolfson \cite{schoen} studied the minimization problem and showed that, when the real dimension is four, a Lagrangian minimizing area among all Lagrangians in a given class exists,  is smooth everywhere except finitely many points, but not necessarily a minimal surface. Later Wolfson \cite{wolfson}  found a  Lagrangian sphere  $\Sigma$ with nontrivial homology on a given K3 surface for which the Lagrangian which minimizes area among all Lagrangians homologous to $\Sigma$ is not a special Lagrangian and the surface which minimizes area  among all surfaces  homologous to $\Sigma$ is not Lagrangian. This shows the subtle nature of the problem and that  variational methods do not seem to be very effective. For this reason there has been increased interest in evolving a given Lagrangian submanifold   by the gradient flow for the area functional  (Lagrangian mean curvature flow) and hope  to obtain convergence  to a special Lagrangian. 

Initially there was a  source of optimism and, under the assumption that the tangent planes of the initial Lagrangian lie in some convex subset of the Grassmanian bundle, Smoczyk, Tsui, and Wang \cite{smoczyk, smoczyk-mt, mt1, mt2}  proved that the Lagrangian mean curvature flow exists for all time and converge to a special Lagrangian. Similar  results were also obtained in the symplectic or graphical setting by Chen, Li, Smoczyk, Tian, Tsui, and Wang \cite{chen, smoczyk, smoczyk2, smoczyk-mt, tsui, Wa1,mt1,mt2,mt3}. Unfortunately, the minimal surfaces which were produced by this method were already known to exist which means that, in order to find new special Lagrangians, one should drop the convexity assumptions on the image of the Gauss map. The drawback in doing so is that long-time existence can no longer be assured and, as a matter of fact, the author showed \cite{neves} that  finite-time singularities do occur for very ``well-behaved'' initial conditions.

\begin{thm} There is $L\subset \C^2$  Lagrangian, asymptotic to two planes at infinity, and with arbitrarily small  oscillation of the Lagrangian angle so that the solution to mean curvature flow develops finite time singularities.
\end{thm}
These examples all live in $\C^2$ and so it was a natural  open question whether, in a compact Calabi-Yau, one could have ``good'' initial conditions which develop finite time singularities under the flow. As a matter of fact,    Thomas and Yau \cite{thomas} proposed a notion of  ``stability'' for the flow (see either \cite{thomas} or \cite{neves3} for the details) and conjectured that Lagrangian mean curvature flow of ``stable'' initial conditions will exist for all time and converges to a special Lagrangian. Unfortunately, their stability condition  is in general hard to check  and it seems to be  a highly  nontrivial statement the existence of Lagrangians which are ``stable'' in their sense and not special Lagrangian. Thus Wang \cite{wang} simplified the Thomas-Yau conjecture to become
\begin{conj} Let $L$ be a Lagrangian in a Calabi-Yau manifold which is embedded and Hamiltonian isotopic to a special Lagrangian  $\Sigma$. Then the Lagrangian mean curvature flow exists for all time and converges to $\Sigma$.
\end{conj}

Schoen and Wolfson \cite{schoen1} constructed solutions to Lagrangian mean curvature flow which become singular in finite time and where the initial condition is homologous to a special Lagrangian. On the other hand,  we remark that the flow {\em does} distinguish between isotopy class and homology class. For instance, on a two dimensional torus, a curve $\gamma$  with a single self intersection which is homologous to a simple closed geodesic will develop a finite time singularity under curve shortening flow while if we make the more restrictive assumption that $\gamma$ is isotopic to a simple closed geodesic, Grayson's Theorem \cite{grayson} implies that the curve shortening flow will exist for all time and sequentially converge to a simple closed geodesic.

To this end, the author has recently shown \cite[Theorem A]{neves3} that Wang's conjecture is false.

\begin{thm}\label{main0} Let $M$ be a four real dimensional Calabi-Yau and $\Sigma$ an embedded Lagrangian. There is $L$ Hamiltonian isotopic to $\Sigma$ so that the Lagrangian mean curvature flow starting at $L$ develops a finite time singularity.
\end{thm}

In any case the upshot is that it will be hard to avoid singularities for Lagrangian mean curvature flow and so it is important to understand how singularities form if one expects to use the flow to produce special Lagrangians. The subject is still in its infancy and so  the purpose of this survey it to collect some the basic techniques that have been used to tackle singularity formation  and exemplify how they can be applied in simple cases. For more on long time existence and convergence results the reader is encouraged to read \cite{wang, wang-2}. 

\noindent {\bf Acknowledgements:} The author would like to express his gratitude to Dominic Joyce for extensive comments that improved tremendously this survey.

\section{Preliminaries}
Let $J$ and $\omega$ denote, respectively, the standard complex structure on $\C^n$ and the standard symplectic
form on $\C^n$. We consider  the closed complex-valued $n$-form given by
$$\Omega\equiv dz_1\wedge\ldots\wedge dz_n$$ and the Liouville form given by
$$
\lambda\equiv\sum_{i=1}^{n}x_idy_i-y_idx_i, \quad d\lambda=2\omega,
$$
where $z_j=x_j+iy_j$ are complex coordinates of $\C^n$. We set
$$B_S=\{x\in\C^n\,|\, |x|<S\}\quad\mbox{and}\quad A(R,S)=\{x\in\C^n\,|\, R<|x|<S\}.$$
Given $f\in C^1(\C^n)$, $Df$ denotes its gradient in $\C^n$ and $\nabla f$ its gradient in  $L$.

A smooth $n$-dimensional submanifold $L$ in $\C^n$ is said to be {\em Lagrangian} if $\omega_L=0$ and a simple computation shows that 
$$\Omega_L=e^{i\theta}\vol_L,$$
where $\vol_L$ denotes the volume form of $L$ and $\theta$ is some multivalued function called the {\em
Lagrangian angle}. When the Lagrangian angle is a single valued function the Lagrangian is called {\em
zero-Maslov class} and if
$$\cos \theta\geq \varepsilon$$ for some positive $\varepsilon$, then $L$ is said to be {\em almost-calibrated}.
Furthermore, if $\theta\equiv\theta_0$, then $L$ is calibrated by
$$\re\,\left( e^{-i\theta_0}\Omega\right)$$ and hence area-minimizing. In this case, $L$ is referred as being
{\em special Lagrangian}.

 For a smooth Lagrangian, the
relation between the Lagrangian angle and the mean curvature is given by the following remarkable property (see
for instance \cite{thomas})
$$H=J\nabla \theta.$$
 A Lagrangian $L_0$ is said to be {\em rational} if for some real number $a$
$$\lambda\left(H_1(L_0, \Z)\right)= \{a2k\pi\,|\,k\in \Z\}.$$ Any Lagrangian having $H_1(L_0, \Z)$ finitely generated
can be perturbed in order to become rational and so this condition is not very restrictive. When $a=0$ the Lagrangian is called {\em exact} and this means there is $\beta\in C^{\infty}(L_0)$ for which $d\beta=\lambda$. Furthermore, if
$L_0$ is also zero-Maslov class, it was shown in  \cite[Section 6]{neves} that the rational condition is preserved by
Lagrangian mean curvature flow.

Let $L_0$ be a smooth Lagrangian in $\C^n$ with  area ratios  bounded above, meaning there is  $C_0$ so that
$$
\H^n\bigl(L_0\cap B_R(x)\bigr)\leq C_0 R^n\mbox{ for all }R>0\mbox{ and }x\in \C^n.
$$
Under suitable conditions, bounded area ratios, Lagrangian, zero-Maslov class, and almost-calibrated are conditions which are preserved by the flow. 
All solutions  to  Lagrangian mean curvature flow considered in this survey are assumed to have polynomial area growth, bounded Lagrangian angle, and a primitive for the Liouville  form with polynomial growth as well.  

A submanifold $L$ of Euclidean space is called a {\em self-expander} if $H=x^{\bot}/2$ and what this means is that $M_t=\sqrt t M$ is a smooth solution to mean curvature flow for all $t>0$. If $L$ is  an exact and  zero-Maslov class Lagrangian in $\C^n$ then 
$$H=\frac{x^{\bot}}{2}\implies 2J\nabla \theta=-J\nabla \beta\implies \nabla(\beta+2\theta)=0$$
and so $\beta+2\theta$ is constant.

Given any $(x_0,T)$ in $\R^{2n}\times \R$, we consider the backwards heat kernel
$$\Phi(x_0,T)(x,t)=\frac{\exp\left(-\frac{|x-x_0|^2}{4(T-t)}\right)}{(4\pi(T-t))^{n/2}}.$$

\section{Basic Techniques}\label{basic}
 
 I will describe the main technical tools that have been used to understand singularities.
 \subsection{White's Regularity Theorem}\label{white1} Let $(M_t)_{t\geq 0}$ be a smooth solution to mean curvature flow of $k$-submanifolds in $\R^n$. Consider the local Gaussian density ratios given by
 $$\Theta_t(x_0,l)=\int_{M_t}\Phi(x_0,l)(x,0)d\H^k.$$
 The following theorem is proven in \cite{white}. Its content is that if the local Gaussian density ratios are very close to one, the submanifolds enjoy a priori estimates on a slightly smaller set. 
 \begin{thm}[White's Regularity Theorem]\label{reg} There are $\varepsilon_0=\varepsilon_0(n,k),$ $C=C(n,k)$ so that if $\partial M_t\cap B_{2R}=\emptyset$ and 
 $$\Theta_t(x,l)\leq 1+\varepsilon_0\quad\mbox{for all}\quad l\leq R^2, x\in B_{2R}, \mbox{ and }t\leq R^2,$$
 then the $C^{2,\alpha}$-norm of $M_t$ in $B_R$ is bounded by $C/\sqrt t$ for all $t\leq R^2$.
\end{thm} 
 
 \subsection{Monotonicity Formulas} 
In \cite{huisken} Huisken proved the following fundamental identity.

 \begin{thm}[Huisken's monotonicity formula]\label{hmf} Let $f_t$ be a smooth family of functions  on $L_t$. Then, assuming all quantities are finite, 
\begin{multline*}
	\frac{d}{dt}\int_{L_t} f_t\Phi(x_0,T)d\H^n=	\int_{L_t}\left( \frac{df_t}{dt}-\Delta f_t \right)\Phi(x_0,T)d\H^n\\
	-\int_{L_t}f_t\left| H+\frac{(\x-x_0)^{\bot}}{2(T-t_0)}\right|^2\Phi(x_0,T)d\H^n.
\end{multline*}
\end{thm}

The next lemma determines test functions to be used in Huisken's monotonicity formula. 
\begin{lemm}\label{test} Let $(L_t)_{t\geq 0}$ be a zero-Maslov class smooth solution to Lagrangian mean curvature flow. Then
\begin{itemize}
\item[i)] There is a smooth  family of functions $\theta_t\in C^{\infty}(L_t)$ such that 
$$H=J\nabla \theta_t\quad\mbox{and}\quad\frac{d}{dt}\theta_t^2=\Delta \theta_t^2-2|H|^2;$$
\item[ii)] Assume that $L_0$ is also exact. There is a smooth family of functions $\beta_t\in C^{\infty}(L_t)$ with $d\beta_t=\lambda$ and
 $$\frac{d}{dt}(\beta_t+2(t-T)\theta_t)^2=\Delta(\beta_t+2(t-T)\theta_t)^2-2|2(t-T)H-x^{\bot}|^2;$$
 \item[iii)] If $\mu\in C^{\infty}(\C^{n})$ is such that the one parameter family of diffeomorphisms $(\phi_s)_{s\geq 0}$ generated by $JD\mu$ is in $SU(n)$, then 
$$ \frac{d}{dt}\mu^2=\Delta\mu^2-2|\nabla \mu|^2;$$
\item[iv)] If $n=2$ and $\mu(z_1,z_2)=x_1y_2-x_2y_1$, then
$$\frac{d}{dt}\mu^2=\Delta\mu^2-2|\nabla \mu|^2;$$

\end{itemize}
\end{lemm}
\begin{rmk}
If $L$ is special Lagrangian, the third identity implies that $\mu$ is harmonic in $L$, a fact which was observed by Joyce in \cite[Lemma 3.4]{joyce}. The geometric interpretation is that $\mu$ is obtained from the moment map of some group action.
\end{rmk}
\begin{proof}
The first two equations can be found in \cite[Section 6]{neves}. We now show the third identity.  It suffices to show that
$$ \frac{d \mu}{dt}=\Delta \mu.$$
For each fixed $t$ consider the family $L_{s,t}=\phi_s(L_t)$. It is simple to see that $L_{s,t}$ is Lagrangian for all $s$ and the Lagrangian angle $\theta_{s,t}$ satisfies (see \cite[Lemma 2.3]{thomas})
$$\frac{d}{ds}\theta_{s,t}=\Delta \mu.$$
On the other hand, each $\phi_s\in SU(n)$, which means that $\theta_{s,t}=\theta_t\circ \phi^{-1}_s$ and thus $\frac{d}{ds}_{|s=0}\theta_{s,t}=-\langle \nabla \theta_t, Z \rangle$. Therefore
$$  \frac{d \mu}{dt}=\langle H, D\mu\rangle=-\langle \nabla \theta_t, Z \rangle= \frac{d}{ds}_{|s=0}\theta_{s,t}=\Delta \mu.$$

To show the last identity one can either argue that the one parameter family of diffeomorphisms  generated by $Z=JD\mu$ is in $SU(2)$ or see directly that, because
each coordinate function evolves by the linear heat equation, we have 
$$\frac{d \mu}{dt}=\Delta \mu-2\langle X_1^{\top},Y^{\top}_2\rangle+2\langle Y^{\top}_1,X^{\top}_2\rangle,$$
where $X_i=Dx_i, Y_i=Dy_i$ for $i=1,2$ and 
\begin{multline*}
\langle X_1^{\top},Y_2^{\top} \rangle-\langle Y_1^{\top},X_2^{\top} \rangle=-\langle (J Y_1)^{\top},Y_2 \rangle-\langle Y_1^{\top},X_2 \rangle\\
=-\langle J Y_1^{\bot},Y_2 \rangle-\langle Y_1^{\top}, X_2 \rangle=-\langle Y_1^{\bot}+Y_1^{\top},X_2 \rangle=-\langle Y_1,X_2\rangle=0.
\end{multline*}
\end{proof}

This lemma can be combined with Theorem \ref{hmf} to show

\begin{cor}\label{cor-test}\noindent
\begin{itemize}
\item[i)] A smooth zero-Maslov class Lagrangian  which is a self-shrinker must be a plane.
\item[ii)] If $(L_t)_{t>0}$ is an exact and smoth zero-Maslov class solution to Lagrangian mean curvature flow with area ratios bounded below and such that $L_{\varepsilon_i}$  converges in the varifold sense to a cone $L_0$ when $\varepsilon_i$ tends to zero then, for all $t>0$, $$L_t=\sqrt t L_1.$$  
\item[iii)] Let $\mu$ be a function satisfying the conditions of Lemma \ref{test} iii) or iv). If $(L_t)_{t>0}$ is a smooth solution to Lagrangian mean curvature flow such that, when $t$ tends to zero,  $L_t$ tends, in the Radon measure sense, to a measure supported in $\mu^{-1}(0)$, then $L_t\subset \mu^{-1}(0)$ for all $t$.
\end{itemize}
\end{cor}
 \begin{rmk}\noindent
 \begin{itemize}
 \item[a)] Assuming almost-calibrated, the first statement was proven by Wang in \cite{Wa1} (see also \cite{chen1} for a similar result in the symplectic case). The second statement was proven in \cite{neves2}.
 \item[b)] It is important in i) that we assume $L$ to have bounded Lagrangian angle and no boundary. Otherwise, as it was pointed out by Joyce, the universal cover of a circle or half circle in $\C$ would be counterexamples. 
 \item[c)] It is important in ii) that we assume $L_t$ to be smooth for all $t>0$ because otherwise the result would not be true. For instance, for curve shortening flow, $\sigma_t$ could be $\{(x,y)\,|\, xy=0\}$ for all $t\leq 2$ and $\sigma_t=\sqrt{t-2}\sigma_3$ for all $t>2$, where $\sigma_3$ is a self-expander asymptotic to $\sigma_2$.
 \end{itemize}
\end{rmk}
\begin{proof}
To prove i) set $L_t=\sqrt {-t}L$  which is a smooth solution to Lagrangian mean curvature flow for $t<0$. Choose $(x_0,T)=(0,0)$ and consider
$$\theta(t)=\int_{L_t} \theta_t^2\Phi(x_0,T)d\H^n.$$
Scale invariance implies that $\theta(t)$ is constant as a function of $t$ and so its derivative must be zero.  Hence, combining Theorem \ref{hmf} with Lemma \ref{test} i) we have that $L$ has $H=0$. Moreover, $L$ is a self-shrinker and so it must  have $x^{\bot}+2H=0$, which means that $L$ is a smooth minimal cone. Thus, $L$ must be a plane.

To prove the second statement note that the function $\beta_t$ can be defined as 
$$\beta_t(x)=\int_{\gamma(p_t,x)}\lambda+\beta_t(p_t),$$
where $p_t$ belongs to $L_t$ and $\gamma(p_t,x)$ is any path in $L_t$ connecting $p_t$ to $x$. 

Because $L_0$ is a varifold with $x^{\bot}=0$ we have that $\lambda=0$ when restricted to $L_0$ and thus, from varifold convergence and the fact  area ratios are bounded  below, we have that when $t_i$ tends to zero $\beta_{t_i}$ converges uniformly to a constant which we can assume to be zero. As a result, we obtain that
$$\gamma(t)=\int_{L_t} (2t\theta_t+\beta_t)^2\Phi(0,1)d\H^n$$
has $\gamma(t_i)$ tending to zero. Furthermore, we have from Theorem \ref{hmf} and Lemma \ref{test} ii) that 
$$\frac{d}{dt}\gamma(t)\leq 	-2\int_{L_t}\left| 2tH-x^{\bot}\right|^2\Phi(0,1)d\H^n$$
which means that $\gamma(t)$ is non-increasing and so it must be zero for all $t$. Therefore $2tH-x^{\bot}=0$ on $L_t$ and this implies $L_t=\sqrt t L_1$.

To show iii) note that from 
Lemma \ref{test} iii) and Theorem \ref{hmf} we have for all $t<T$
$$\frac{d}{dt}\int_{L_t} \mu^2\Phi(0,T)d\H^n\leq 0.$$
The result follows because 
$$\lim_{t\to 0}\int_{L_t} \mu^2\Phi(0,T)d\H^n=0.$$

\end{proof}

\subsection{Poincar\'e type Lemma} In order to study blow-ups of singularities it is important to have a criteria which implies that a function $\alpha_i$ on $N^i$ with $L^2$ norm of the gradient converging to zero must converge to a constant. It is simple to construct a sequence $N^i$ (not necessarily Lagrangian) converging (in some suitable weak sense) to a disjoint union of two spheres $S_1, S_2$ and a sequence of functions $\alpha_i$ with $L^2$ norm of the gradient converging to zero so that $\alpha_i$ tends to $1$ on $S_1$ and $-1$ on $S_2$. 
The next proposition gives
conditions which rule out this possibility. 
\begin{lemm}\label{cont}
Let $(N^i)$ and $(\alpha_i)$ be a sequence of smooth $k$-submanifolds in $\R^n$  and  smooth functions on $N^i$
respectively, such that 
  the following properties hold for some $R>0$:
\begin{itemize}
\item[a)] There exists a constant $D_0$ such that  $$\H^k(N^i\cap B_{3R}))\leq D_0R^k\mbox{ for all }i\in\N$$
and  
$$\left(\H^k(A)\right)^{(k-1)/k}\leq D_0 \H^{k-1}(\partial A)$$
for every open subset $A$ of $N^i\cap B_{3R}$ with rectifiable boundary.
\item[b)] There exists a constant $D_1$ such that  for all $i \in \N$
$$\sup_{N^i\cap B_{3R}}|\nabla \alpha_i|+R^{-1}\sup_{N^i\cap B_{3R}}|\alpha_i|\leq D_1.$$ 
\item[c)]$$\lim_{i \to \infty}\int_{N^i\cap B_{3R}}|\nabla \alpha_i|^2 d\H^k=0.$$
\item[d)] $\partial N^i \cap B_{3R}=0$ and   $N^i\cap B_{2R}$ contains only one connected component which intersects $B_R$.
\end{itemize}
There is   $\alpha$ such that, after passing to a subsequence, $$\lim_{i\to\infty}\sup_{N^i\cap B_R} |\alpha_i-\alpha|=0.$$
\end{lemm}

 \begin{rmk}
 A version of this lemma with stronger hypothesis was proven in \cite[Proposition A.1]{neves}. Hypothesis a) is needed so that we have some  control on the sequence $N^i$. Note that it rules out the example, described above, of $N^i$ degenerating into two spheres. Hypothesis b) is also needed because if $N^i=\{(z,w)\in \C^2\,|\, zw=1/i\}$, it is not hard to construct a sequence $\alpha_i$ for which c) is true but $\alpha_i$ does not tend to a constant function. 
 Finally, the last hypothesis is needed because otherwise the lemma would fail for trivial reasons. 
 \end{rmk}

\begin{proof} 
Throughout  this proof, $K=K(D_0,D_1,k)$ denotes a generic constant depending only on the mentioned
quantities. Choose any sequence $(x_i)$ in $N^i\cap B_R$. After passing to a subsequence, we have  $$
\lim_{i \to\infty}x_i=x_0\quad\mbox{and}\quad\lim_{i \to
  \infty}\alpha_i(x_i)=\alpha
$$ for some $x_0 \in B_R$ and $\alpha \in \R$. Furthermore, consider 
a sequence $(\varepsilon_j)$ converging to zero and define
$$ N^{i,\alpha,j}=\alpha_i^{-1}([\alpha-\varepsilon_j,\alpha+\varepsilon_j]).$$
The sequence $(\varepsilon_j)$ can be chosen so that,  for all $j\in \N$,
$$
\lim_{i \to \infty}\H^{k-1}\bigl( \partial N^{i,\alpha,j}\cap B_{3R}\bigr)=0
$$
because, by the coarea formula, we have
\begin{multline*}
\lim_{i \to \infty}\int_{-\infty}^{\infty}\H^{k-1}\bigl( \{\alpha_i=s\}\cap B_{3R}\bigr)ds
=\lim_{i \to \infty}\int_{N^i\cap B_{3R}}|\nabla \alpha_i|d\H^{k}\\
\leq \lim_{i \to \infty}K R^{k/2}\left(\int_{N^i\cap B_{3R}}|\nabla \alpha_i|^2d\H^{k}\right)^{1/2} =0.
\end{multline*}

\begin{lemm}\label{below-area}
For every $j \in \N$ $$ \liminf_{i\to\infty}\H^k\bigl(N^{i,\alpha,j}\cap B_R(x_0)\bigr)\geq K R^k.
$$
\end{lemm}
\begin{proof}
Given $y_i\in N^i$, denote by $\hat{B}_{r}(y_i)$ the {\em intrinsic} ball  in $N^i$ of radius $r$.  We start by showing that $\H^{k}(\hat{B}_{r}(y_i))\geq Kr^k$ for all $y_i\in B_{2R}\cap N^i$ and $r<R$.
Set
$$\psi(r)=\H^k\left(\hat{B}_{r}(x_i)\right)$$ which has, for all $r<R$, derivative given by
$$\psi^{\prime}(r)=\H^{k-1}\left(\partial \hat{B}_{r}(y_i)\right)\geq K (\psi(r))^{(k-1)/k}.$$
Hence, integration implies 
$\psi(r)\geq K r^k$ and the claim follows. From hypothesis b) there is $s_j=s(j,k,D_0,D_1,R)<R$ such that, for all $i$ sufficiently large, $\hat B_{s_j}(x_i)\subset N^{i,\alpha,j}$ and thus 
$$\H^{k}\left(B_{s}(x_i)\cap N^{i,\alpha,j}\right)\geq Ks^k\quad \mbox{for all }s\leq s_j.$$
Set
$$
\psi_{i,j}(s)=\H^k\bigl(N^{i,\alpha,j}\cap B_s(x_i)\bigr)
$$
which has, by the coarea formula, derivative satisfying 
\begin{align*}
\psi_{i,j}^{\prime}(s)& = \oint_{\partial B_s(x_i)\cap N^{i,\alpha,j}}\frac{|x-x_i|}{|(x-x_i)^{\top}|}d\H^{k-1} \geq \H^{k-1}\left(\partial B_s(x_i)\cap N^{i,\alpha,j}\right)\\
& = \H^{k-1}\left(\partial \left(B_s(x_i)\cap N^{i,\alpha,j}\right)\right)-\H^{k-1}\left( B_s(x_i)\cap\partial N^{i,\alpha,j}\right)\\
&\geq  K\left(\H^{k}\left(B_s(x_i)\cap N^{i,\alpha,j}\right)\right)^{(k-1)/k}-\H^{k-1}\left(\partial N^{i,\alpha,j}\cap B_{3R}\right)\\
&=K\left(\psi_{i,j}(s)\right)^{(k-1)/k}-\H^{k-1}\left(\partial N^{i,\alpha,j}\cap B_{3R}\right)
\end{align*}
for almost all $s$. Integration implies 
$$
\psi_{i,j}^{1/k}(R)\geq K(R-r_j)-\H^{k-1}\left(\partial N^{i,\alpha,j}\cap B_{3R}\right)\int_{r_j}^R\psi_{i,j}^{(1-k)/k}(t)dt,
$$
where $r_j=\min\{s_j,KR/2\}$. Note the integral term is bounded independently of $i$ for all $i$ sufficiently large and so
$$\liminf_{i\to\infty}\psi_{i,j}^{1/k}(R)\geq K(R-r_j)\geq KR/2.$$
This proves  Lemma \ref{below-area}.

\end{proof}
Suppose there is $y_i\in N^i\cap B_R$ converging to $y_0\in B_R$ so that  $\alpha_i(y_i)$ tends to $\bar \alpha$ distinct from $\alpha$. 
Repeating the same type of arguments we find a closed interval $I$ disjoint from
$[\alpha-\varepsilon_j,\alpha+\varepsilon_j]$ such that, after passing to a subsequence,
$$
\lim_{i \to \infty}\H^k\bigl(\alpha_i^{-1}(I)\cap B_R(y_0)\bigr)\geq K R^k.
$$
Given any positive integer $p$, pick disjoint closed intervals
$$I_1,\cdots,I_p$$ lying between $I$ and
$[\alpha-\varepsilon_j,\alpha+\varepsilon_j]$. Hypothesis d) implies that, for all $i$ sufficiently large, 
$\alpha_i^{-1}(I_l)\cap B_{2R}$ is not empty. Hence, arguing as before, we find
$y_1,\dotsc, y_p$ in $B_{2R}$
 such that, after passing to a subsequence, $$
 \lim_{i \to \infty}\H^k\bigl(\alpha_i^{-1}(I_l)\cap B_R(y_l)\bigr)\geq K R^k,
$$
for all $l$ in $\{1,\dotsc,p\}$. This implies 
\begin{align*}
\lim_{i \to \infty}\H^k\bigl(N^i\cap B_{2R}\bigr) & \geq \lim_{i \to
  \infty}\sum_{l=1}^p\H^k\bigl(\alpha_i^{-1}(I_j)\cap B_R(y_l)\bigr)\\
&\geq pKR^k.
\end{align*}
Choosing $p$ sufficiently large we get a contradiction. This proves Lemma \ref{cont}.
\end{proof}

The next result gives conditions which guarantee   Lemma \ref{cont} a)  holds.

\begin{lemm}\label{aa} Let $L$ be a Lagrangian in $\C^n$ such that $\partial L \cap B_R=\emptyset$ and either i)
$$\inf_{L \cap B_R}\cos \theta\geq \delta$$
or ii) $n=2$ and for some $\varepsilon$ small enough
$$\int_{L \cap B_R}|H|^2d\H^2\leq \varepsilon.$$ 

There is $D=D(\delta,n)$ so that $$\left(\H^n(A)\right)^{(n-1)/n}\leq D \H^{n-1}(\partial A)$$
for all open subsets $A$ of $L\cap B_{R}$ with rectifiable boundary.
\end{lemm}
\begin{proof} We follow \cite[Lemma 7.1]{neves} and prove i).
The Isoperimetric Theorem \cite[Theorem 30.1]{Leon} guarantees the existence of an integral current $B$ with
compact support such that $\partial B=\partial A$ and for which
$$
\left(\H(B)\right)^{(n-1)/n}\leq C \H^{n-1}(\partial A),
$$
where $C=C(n)$. If $T$ denotes the cone over the current $A-B$ (see \cite[page 141]{Leon}), then $\partial
T=A-B$ and thus, because $$\re \,\Omega_{L}=\cos \theta \geq \delta,$$ we obtain
\begin{align*}
\H^{n}(A) & \leq \delta^{-1}\int_A \re\, \Omega=\delta^{-1}\int_B \re\, \Omega+\delta^{-1}\int_{\partial T}\re\, \Omega\\
& \leq \delta^{-1}\H^n(B)+\delta^{-1}\int_T d\re\,\Omega\leq \delta^{-1}\left(C\H^{n-1}(\partial
A)\right)^{n/(n-1)}.
\end{align*}
To prove ii) we use  Michael-Simon Sobolev inequality which implies
(see \cite[Theorem 18.6]{Leon})
$$
\left(\H^2(A)\right)^{1/2}\leq C\int_A |H| +C\H^{1}(\partial A)
$$
for some universal constant $C$. In this case we have
$$
\left(\H^2(A)\right)^{1/2}\leq C\left(\H^2(A)\right)^{1/2} \left(\int_A |H|^2\right)^{1/2}+C\H^{1}(\partial A)
$$
and so we get the desired result whenever $$C^2 \int_{L\cap B_R} |H|^2\leq 1/4.$$
\end{proof}

\subsection{Compactness Result} We state a compactness result for zero-Maslov class Lagrangians with bounded Lagrangian angle. The proof can be found in \cite[Proposition 5.1]{neves}. 

\begin{prop}\label{general}
Let $L^i$ be a sequence of smooth zero-Maslov class Lagrangians in $\C^n$ such that, for some fixed $R>0$, the
following properties hold:
\begin{itemize}
\item[(a)] There exists a constant $D_0$ for which $$\H^n(L^i\cap B_{2R})\leq
D_0R^n\quad\mbox{and}\quad\sup_{L^i\cap B_{2R}}|\theta_i|\leq D_0$$ for all $i \in \N$.

 \item[(b)]
$$\lim_{i \to \infty}\H^{n-1}(\partial L^i\cap B_{2R}(0))=0$$
and $$\lim_{i \to \infty}\int_{L^i\cap B_{2R}(0)}|H|^2\a=0.$$
\end{itemize}
Then, there exist a finite set $\{\bar\theta_1,\ldots,\bar\theta_N\}$ and integral special Lagrangians currents
$$L_1,\ldots,L_N$$
such that, after passing to a subsequence, we have for every smooth function $\phi$ compactly supported in
$B_R(0)$ and every $f$ in $C(\R)$
$$
\lim_{i \to \infty}\int_{L^i}f(\theta_{i})\phi\a=\sum_{j=1}^N m_j f(\bar\theta_j)\mu_j(\phi),
$$
where $\mu_j$ and $m_j$ denote, respectively, the Radon measure of the support of $L_j$ and its multiplicity.
\end{prop}

\begin{rmk}
With the extra assumption that $L^i$ is almost-calibrated, a similar result to Proposition  \ref{general} was proven in \cite[Theorem 4.1]{CL}. 
The proposition is optimal in the sense that given Lagrangians planes $P_1$, $P_2$ intersecting transversely at the origin and two positive integers $n_1$, $n_2$ it is possible to construct a sequence of zero-Malsov class Lagrangians $L^i$ with $L^2$ norm of mean curvature converging to zero and such that $L^i$ tends to $n_1P_1+n_2P_2$ in the varifold  sense.
\end{rmk}

\section{Applications I: Blow-ups}\label{blowup}

 Let $(L_t)_{0\leq t<T}$ be a  zero-Maslov class solution to Lagrangian mean curvature flow in $\C^n$ with  a singularity at  $x_0$ at time $T$. Pick a sequence $(\lambda_i)_{i\in\N}$ tending to infinity and consider the sequence of blow-ups
$$L^i_s=\lambda_i(L_{T+s\lambda_i^{-2}}-x_0)\quad\mbox{for all }s<0.$$ 
The next theorem was proven in \cite[Theorem A]{neves} and in \cite{CL} assuming an extra almost-calibrated condition.
\begin{thm}\label{type1}There exist  integral
special Lagrangian current cones
$$L_1,\ldots,L_N$$
with Lagrangian angles   $\{\bar\theta_1,\ldots,\bar\theta_N\}$ 
such that, after passing to a subsequence, we have for every smooth function $\phi$ compactly supported, every
$f$ in $C^2(\R)$, and every $s<0$
$$
\lim_{i \to \infty}\int_{L^i_s}f(\theta_{i,s})\phi\a=\sum_{j=1}^N m_j f(\bar\theta_j)\mu_j(\phi),
$$
where $\mu_j$ and $m_j$ denote the Radon measure of associated with $L_j$ and its multiplicity respectively.

Furthermore, the set $\{\bar\theta_1,\ldots, \bar\theta_N\}$ does not depend on the sequence of rescalings
chosen.
\end{thm}
When $n=2$  special Lagrangian cones are simply a union of planes having the same Lagrangian angle.
\begin{proof}[Sketch of proof]
Set
$$\Theta_i(s)=\int_{L^i_s} \Phi(0,0)d\H^n=\int_{L_{T+s\lambda_i^{-2}}}\Phi(0,T)d\H^n $$
and
$$\theta_i(s)=\int_{L^i_s} \theta_{s}^2\Phi(0,0)d\H^n=\int_{L_{T+s\lambda_i^{-2}}} \theta_{T+s\lambda_i^{-2}}^2\Phi(0,T)d\H^n.$$
From Theorem \ref{hmf} we have for $b<a<0$
\begin{equation}\label{x-zero}
\int^a_b\int_{L^i_s}\left| H-\frac{x^{\bot}}{s}\right|^2\Phi(0,0)d\H^nds=\Theta_i(a)-\Theta_i(b)
\end{equation}
and
\begin{equation}\label{h-zero}
2\int^a_b\int_{L^i_s}\left| H\right|^2\Phi(0,0)d\H^nds\leq\theta_i(a)-\theta_i(b).
\end{equation}
But 
$\int_{L_{t}}\Phi(0,T)d\H^n\quad\mbox{and}\quad \int_{L_{t}}\theta^2_t\Phi(0,T)d\H^n$
are monotone non-increasing by Theorem \ref{hmf} and thus
\begin{align*}
&\lim_{i\to\infty}\Theta_i(a)=\lim_{t\to T}\int_{L_{t}}\Phi(0,T)d\H^n=\lim_{i\to\infty}\Theta_i(b)\\
&\lim_{i\to\infty}\theta_i(a)=\lim_{t\to T}\int_{L_{t}}\theta^2_t\Phi(0,T)d\H^n=\lim_{i\to\infty}\theta_i(b).
\end{align*}
Therefore, we obtain from \eqref{x-zero} and \eqref{h-zero} that
\begin{equation}\label{limit-cone}
\lim_{i\to\infty}\int^a_b\int_{L^i_s}\left(\left| H\right|^2+|x^{\bot}|^2\right)\Phi(0,0)d\H^n ds=0.
\end{equation}
The result follows from combining Proposition \ref{general} with some standard facts of mean curvature flow. 
\end{proof}

When the initial condition is  rational we obtain extra structure regarding the behavior of blow-ups. 

\begin{thm} \label{thmb}Assume the initial condition is rational and, in case $n>2$, almost-calibrated. Then, for almost all $s_0$, if $\Sigma_i\subseteq L^i_{s_0}$ has $\partial \Sigma_i\cap B_{3R}=\emptyset$ and only one connected component of $\Sigma_i\cap B_{2R}$ intersects $B_R$  then, after passing to a subsequence, we can find $j\in \{1,\ldots,N\}$ so that
$$
\lim_{i \to \infty}\int_{\Sigma^i}f(\theta_{i,s_0})\phi d\H^n=m f(\bar\theta_j)\mu_j(\phi),
$$
for every $f$ in $C^2(\R)$ and every smooth $\phi$ compactly supported in $B_{R}(0)$, where $m\leq m_j$ and $\mu_j$  denotes
the Radon measure associated with the  special Lagrangian  cone $L_j$ given by Theorem \ref{type1}
\end{thm}

This theorem is slightly different from the one stated in \cite[Theorem B]{neves} but the proof is identical. We sketch the main idea.
\begin{proof}[Sketch of proof]
For simplicity we assume the initial condition is exact.  Recall that $\left|\nabla \beta_{i,s}\right|=\left|x^{\bot}\right|$ and so, without loss of generality (see \eqref{limit-cone}), we can assume that, when $s=s_0$ or $s= -1$,  
$$\lim_{i\to\infty}\int_{L^i_{s}}\left(\left| H\right|^2+\left|\nabla \beta_{i,s}\right|^2\right)\Phi(0,0)d\H^n=0.$$
We now study the sequences $\Sigma^i$ and $L^i_{-1}$.

From Proposition \ref{general}, we have that $\Sigma^i\cap B_{2R}$ converges in the varifold sense to a stationary varifold $\Sigma$ with Radon measure $\mu_{\Sigma}$,  which can be represented as a sum of special Lagrangian cones with multiplicities. Furthermore in virtue of Lemmas \ref{cont} and \ref{aa} we conclude the existence of $\bar\beta$ so that
$$ \lim_{i \to \infty}\int_{\Sigma^i}f(\beta_{i,s_0})\phi\a= f(\bar \beta) \mu_{\Sigma}(\phi)
$$
for every $f$ in $C^2(\R)$ and every smooth function $\phi$ compactly supported in $B_R$.

Similar ideas to the ones use to prove Proposition \ref{general} (see \cite[Lemma 7.2]{neves} for details) show the existence of sets $\{\bar \theta_1,\ldots,\bar \theta_Q\}$, $\{\bar \beta_1,\ldots,\bar \beta_Q\}$, special Lagrangian cones $L_1,\ldots L_Q$, and integers $m_1,\ldots,m_Q$ so that
for every smooth function $\phi$ compactly supported and  every
$f$ in $C^2(\R)$
$$
\lim_{i \to \infty}\int_{L^i_{-1}}f(\bar \beta_{i,-1}-2(s_0+1)\bar \theta_{i,-1})\phi\a=\sum_{j=1}^Q m_j f(\bar \beta_{j}-2(s_0+1)\bar \theta_{j})\mu_j(\phi),
$$
where $\mu_j$  denotes the Radon measure of associated with $L_j$. Moreover, we can arrange things so that the pairs $(\bar\theta_j,\bar\beta_j)$ are all distinct and thus assume, without loss of generality,  the numbers $\bar \beta_j-2(s_0+1)\bar\theta_j$ are all distinct as well.

We now finish the proof.  Let $f\in C^2(\R)$ be a nonnegative cut off function which is one in $\bar \beta$ and zero at all but at most one element of $$\{\bar \beta_j-2(s_0+1)\bar\theta_j\}_{j=1}^Q$$ and $\phi$ a nonnegative function with compact support in $B_R$.

We have from \eqref{limit-cone} that 
$$\lim_{i\to\infty}\int_{s_0}^{-1}\int_{L^i_s}\left(\left| H\right|^2+|\nabla (\beta_{i,s}+2(s-s_0)\theta_{i,s})|^2\right)\Phi(0,0)d\H^n ds=0$$
and so, using the evolution equation satisfied $\beta_{i,s}+2(s-s_0)\theta_{i,s}$ (see Lemma \ref{test}), it is  not hard to conclude that
\begin{multline*}
\lim_{i \to \infty}\int_{L^i_{s_0}}f(\bar \beta_{i,s_0})\phi\a=\lim_{i \to \infty}\int_{L^i_{-1}}f(\bar \beta_{i,-1}-2(s_0+1)\bar \theta_{i,-1})\phi\a\\
=\sum_{j=1}^Q m_j f(\bar \beta_{j}-2(s_0+1)\bar \theta_{j})\mu_j(\phi).
\end{multline*}
Therefore
\begin{multline*}
\mu_{\Sigma}(\phi)=\lim_{i\to\infty}\int_{\Sigma_i}\phi\a=\lim_{i\to\infty}\int_{\Sigma_i}f(\beta_{i,s_0})\phi\a \\
\leq\lim_{i\to\infty}\int_{L^i_{s_0}}f(\beta_{i,s_0})\phi\a=\sum_{j=1}^Q m_j
f(\bar \beta_{j}-2(s_0+1)\bar \theta_{j})\mu_j(\phi).
\end{multline*}
Because $\mu_{\Sigma}(\phi)>0$ we have that $\bar \beta =\bar \beta_{j_0}-2(s_0+1)\bar \theta_{j_0}$ for a unique $j_0$ and thus the inequalities above become
$$\mu_{\Sigma}(\phi)\leq m_{j_0}\mu_{j_0}(\phi)$$
for all $\phi\geq 0$  with compact support in $B_R$. This implies $\Sigma=mL_{j_0}$ for some $m\leq m_{j_0}$ in $B_R$, and the rest of the proof  follows easily 
\end{proof}

The previous theorem does imply non-trivial statements regarding the blow-ups of singularities. We sketch one simple application, the details of which will appear elsewhere.
\begin{cor}\label{ap1}
Assume the initial condition  is rational and $n=2$. The blow-up limit at a  singularity cannot be two planes $P_1$, $P_2$ each with multiplicity one, distinct Lagrangian angles, and intersecting transversely at the origin, i.e., in Theorem \ref{type1} the case $N=2$, $m_1=m_2=1$, $P_1\cap P_2=\{0\}$, and $\bar \theta_1\neq \bar \theta_2$ does not occur.

\end{cor}
\begin{proof}[Sketch of proof]
We argue by contradiction and sssume $L^i_s$ converges  to $P_1+P_2$ for all $s<0$. 
There is $R_0$ sufficiently large so that for every $0\leq l\leq 4$  and $|x_0|>R_0/2$ we have
$$\int_{P_1+P_2}\Phi(x_0,l)(x,0)d\H^2\leq 1+\varepsilon_0/2$$
and thus, for all $i$ sufficiently large, all $-2\leq s<0$, and all $0\leq l\leq 2$, we also have
$$\Theta^i_{s}(x_0,l)\leq \Theta^i_{-2}(x_0,l+2+s)\leq 1+\varepsilon_0,$$
where the first inequality follows from Theorem \ref{hmf}. Thus, we obtain from White's Regularity Theorem \ref{white1} that for any $K$ large enough and $i$ sufficiently large, we have uniform $C^{2,\alpha}$ bounds for $L^i_s$ on the annulus $A(R_0,K)$ for all $-1\leq s<0$. Some extra work  shows that, on the region $A(R_0,K)$ and  for all $-1\leq s<0$,  $L^i_s$ can be decomposed into two connected components $\Sigma_{1,s}^i, \Sigma_{2,s}^i$ where $\Sigma^i_{j,s}$ is  graphical over $P_j\cap A(R_0,K)$, with $j=1,2$. 

We argue that $L^i_s\cap B_K$ must have two connected components for almost all $-1\leq s<0$. Otherwise we could apply Theorem \ref{thmb} and conclude that the Lagrangian angle of $P_1$ must be identical to the Lagrangian angle of $P_2$. 

Some extra, but standard work, shows that $L^i_s\cap B_K$ can be decomposed into two connected components $\Sigma_{1,s}^i$ and $\Sigma_{2,s}^i$ where $\Sigma^i_{j,s}$ converges in the Radon measure sense to $P_j\cap B_K$.  Hence, each $\Sigma^i_{j,s}$ is very close in the Radon measure sense to a multiplicity one disk. We can then apply White's Regularity Theorem \ref{white1} to each $(\Sigma^i_{j,s})_{-1\leq s<0}$ and conclude, in a smaller ball centered at the origin,  uniform bounds on the second fundamental form of $\Sigma^i_{j,s}$ for all $-1/2\leq s<0$ and all $i$ sufficiently large. This implies uniform bounds for the second fundamental form of $L_t$ in a neighborhood of the origin for all $t<T$ and hence no singularity occurs there.
\end{proof}

\section{Applications II: Self-Expanders}

Recently, Joyce, Lee, and Tsui \cite{joyce} proved the following general existence theorem.
\begin{thm}[Joyce, Lee, Tsui]\label{jlt} Given any two Lagrangian planes $P_1$, $P_2$ in $\C^n$ such that neither $P_1+P_2$ nor $P_1-P_2$ are  area-minimizing , there is a Lagrangian self-expander $L$ which is exact, zero-Maslov class with bounded Lagrangian angle, and asymptotic to $P_1+P_2$, meaning $\sqrt tL$ converges, as Radon measures, to $P_1+P_2$ when $t$ tends to zero.
\end{thm}
\begin{rmk}\
\begin{itemize}
\item[i)] The self-expander $L$ found in \cite{joyce} is explicit. 
\item[ii)] In \cite{anciaux}, Anciaux found such examples assuming $L$ is invariant under a certain $SO(n)$ action. In this case the self-expander equation  reduces to an O.D.E.
\end{itemize}
\end{rmk}



The next theorem shows that self-expanders are attractors for the flow in $\C^2$.  The ideas for the proof are taken from \cite[Section 4]{neves3} where a slightly more general version is proven. 

Pick two Lagrangian planes $P_1$, $P_2$ in $\C^2$ so that $P_1\pm P_2$ is not area minimizing and $P_1\cap P_2=\{0\}$. Assume $(L_t)_{t\geq 0}$ is an exact,  zero-Maslov class, almost-calibrated smooth solution to Lagrangian mean curvature flow in $\C^2$. 

\begin{thm}\label{sex}
Fix $S_0$ and $\nu$. There are $R_0$ and $\delta$ so that if  $L_0$ is $\delta$-close in $C^{2,\alpha}$ to $P_1+P_2$ in $A(\delta, R_0)$, then, for all $1\leq t\leq 2$,  $t^{-1/2}L_t$ is $\nu$-close in $C^{2,\alpha}(B_{S_0})$ to a smooth self-expander $Q$ asymptotic to $P_1+P_2$.
\end{thm}

\begin{rmk}\
\begin{enumerate}
\item[i)]The content of the theorem is that if the initial condition is very close,  in a precise sense, to a non area-minimizing configuration of two planes and the flow exists smoothly for all $0\leq t\leq 2$, then the flow will be very close to a smooth self-expander for all $1\leq t\leq 2$.
\item[ii)] The result is false if one removes  the hypothesis that the flow exists smoothly for all $0\leq t\leq 2$. For instance, there are known examples \cite[Theorem 4.1]{neves} where $L_0$ is very close to $P_1+P_2$ and a finite-time singularity happens at a very short time $t_1$.  In this case $L_{t_1}$ can be seen as a transverse intersection of small perturbations of $P_1$ and $P_2$ (see \cite[Figure 2]{neves}) and we could continue the flow past the singularity by flowing each component of $L_{t_1}$ separately, in which case $L_1$ would be very close to $P_1+P_2$ and this is not a {\em smooth} self-expander.
\item[iii)] The smoothness assumption enters the proof in Lemma \ref{stationary}. The key fact is that if $L_0\cap B_R$ is connected and the flow exists smoothly, then $L_t\cap B_R$ will also be connected for all $0\leq t\leq 2$ (this fails in the example described above).
\end{enumerate}
\end{rmk}
\begin{proof}[Sketch of proof]
Consider  a sequence $(R^i)$ converging to infinity,  a sequence $(\delta_i)$ converging to zero, and a sequence of smooth flows $(L^i_t)_{0\leq t\leq 2}$ satisfying the theorem's hypothesis  with $R_0=R^i$, $\delta=\delta^i$. From  compactness for integral Brakke motions \cite[Section 7.1]{ilmanen1} we know that, after passing to a subsequence, $(L^i_t)_{0\leq t\leq 2}$ converges to an integral Brakke motion $(\bar L_t)_{0\leq t\leq 2}$, where $L^i_0$  converges to $P_1+P_2$. 

Because $L^i_0$ converges to $P_1+P_2$ we can assume, without loss of generality, that 
$$\lim_{i\to\infty}\int_{L^i_0}(\beta^i_{0})^2\Phi(0,4)d\H^2=0.$$ 
Thus, from Theorem \ref{hmf} and Lemma \ref{test} ii), we get that for every  $0<s<4$
\begin{multline}\label{exex}
\lim_{i\to\infty}\int_0^s\int_{L^i_t}\left| 2tH-x^{\bot}\right|^2\Phi(0,4)d\H^2+\lim_{i\to\infty}\int_{L^i_s}(\beta^i_{s}+2s\theta^i_s)^2\Phi(0,4)d\H^2\\
\leq \lim_{i\to\infty}\int_{L^i_0}(\beta^i_{0})^2\Phi(0,4)d\H^2=0,
\end{multline}
which means that $H={x^{\bot}}/{2t}\quad\mbox{ on $\bar L_t$ for all $t>0$}$ and thus  $\bar L_t=\sqrt t \bar L_1$ as varifolds for every $t>0$ (see proof of \cite[Theorem 3.1]{neves2}). Moreover, some technical work \cite[Lemma 4.4]{neves3} shows  that $\bar L_t$ converges as Radon measures to $P_1+P_2$ as $t$ tends to zero i.e., $\bar L_1$ is asymptotic to $P_1+P_2$. We are left to show that $\bar L_1$ is smooth.
\begin{lemm}\label{stationary}
$\bar L_1$ is not stationary. 
\end{lemm}
\begin{proof}
	If true, then $\bar L_1$ needs to have $H=x^{\bot}=0$ and thus  $\bar L_t=\bar L_1$ for all $t$, which means (making $t$ tend to zero)  that $\bar L_1=P_1+P_2$.
	We will argue that $\bar L_1$ must be a special Lagrangian, which contradicts the choice of $P_1$ and $P_2$.
	
	Pick $K$ large enough. Because $L^i_0\cap B_{2K}$ is connected and the flow exists smoothly, we claim $L^i_{1}\cap B_{2K}$ has only one connected component intersecting $B_K$. The details can be seen in \cite[Theorem 3.1, Lemma 4.5]{neves3} but the basic idea is to use the fact that  $L_0^i$ very close to $P_1+P_2$ in $A(K/2,3K)$ and so, like in Corollary \ref{ap1},  we conclude  that for all $x_0\in A(K/2,3K)$, all $i$ sufficiently large, all $0\leq t\leq 1$, and all $0\leq l\leq 1$, we  have
$$\Theta^i_{t}(x_0,l)\leq  1+\varepsilon_0.$$
White's Regularity Theorem implies we can control the  $C^{2,\alpha}$-norm of $L^i_t$ on $A(K,2K)$  and some long, but straightforward work,  implies the desired claim.

From varifold convergence we have
			$$\lim_{i\to\infty}\int_{0}^2\int_{L^i_t}|x^{\bot}|^2\Phi(0,4)d\H^2 dt=0$$
			which combined with \eqref{exex} implies that, without loss of generality,  
			\begin{equation*}\label{zero}
			\lim_{i\to\infty}\int_{L_1^i} (|H|^2+|x^{\bot}|^2)\Phi(0,4)d\H^2=0.
			\end{equation*}
Because  $L^i_{1}\cap B_{2K}$ has only one connected component intersecting $B_K$, we can use Lemma \ref{cont} and Lemma \ref{aa} to conclude the existence of  $\bar \beta$ so that, after passing to a subsequence, 
		$$\lim_{i\to\infty}\int_{L^i_1\cap B_K}(\beta^i_1-\bar\beta)^2\phi d\H^2=0.$$
Hence, from  \eqref{exex},  we obtain 
$$\lim_{i\to\infty} \int_{L^i_1\cap B_K}(\bar\beta+2\theta_1^i)^2d\H^2=\lim_{i\to\infty} \int_{L^i_1\cap B_K}(\beta^i_1+2\theta_1^i)^2d\H^2=0$$
which means $\bar L_1$ must be a special Lagrangian cone with Lagrangian angle $-\bar \beta/2$.
\end{proof}

\begin{lemm}\label{small}
There is  $C$ so that 
	$$\int_{\bar L_1}\Phi(y,l)(x,0)d\H^2\leq 2-1/C \quad\mbox{ for every } l\leq 2,\mbox{ and }y\in \R^4.$$
\end{lemm}
\begin{proof}
The details can be found in \cite[Lemma 4.6]{neves3}. Because $\bar L_0=P_1+P_2$  we obtain from Theorem \ref{hmf} 
\begin{multline*}
	\int_{\bar L_1}\Phi(y,l)(x,0)d\H^2+\int_{0}^1\int_{\bar L_t}\left|H+\frac{(x-y)^{\bot}}{2(l+1-t)}\right|^2\Phi(y,l+1-t)(x,0)d\H^2 dt\\
	= \int_{P_1+P_2}\Phi(y,l+1)(x,0)d\H^2\leq 2.
\end{multline*}
The fact that $\bar L_1$ is not stationary allows us to estimate the second term on the first line and  find a constant $C$ such that
$$ \int_{\bar L_1}\Phi(y,l)(x,0)d\H^2\leq 2-1/C$$
for all $y$ and $l\leq 2$.
\end{proof}

The same ideas  used to show Theorem \ref{type1} can be modified to argue the tangent cone at any point $y \in \bar L_1$ must be a union of Lagrangian planes with possible multiplicities. The previous lemma implies it must be a plane with multiplicity one because otherwise
$$\lim_{r\to 0}\int_{\bar L_1}\Phi(y,r^2)(x,0)d\H^2 \geq 2.$$
The mean curvature of $\bar L_1$ satisfies $H=x^{\bot}/2$ and so Allard Regularity Theorem implies uniform $C^{2,\alpha}$ bounds for $\bar L_1$. Therefore, $\bar L_1$ is a smooth  self-expander asymptotic to $P_1+P_2$. Some extra work  shows  $L^i_t$ converges strongly to $\sqrt t\bar L_1$ and this finishes the proof.
\end{proof}

\section{Application III: Stability of Singularities}\label{sta.sta}

We prove a result which is related to \cite[Theorem A]{neves3} but, before we state it,  we need to introduce some notation. 

Given any curve $\gamma:I\longrightarrow \C$, we  obtain a Lagrangian surface in $\C^2$ given by
\begin{equation*}\label{equi}
N=\{(\gamma(s)\cos \alpha,\gamma(s)\sin \alpha)\,|\, s\in I, \alpha \in S^1\}.
\end{equation*}

Any Lagrangian which has the same $SO(2)$ symmetry as $N$   is called {\em equivariant}. If $\mu=x_1y_2-y_1x_2$, it is simple to see that $L$ is equivariant if and only if $L\subset \mu^{-1}(0)$ (see \cite[Lemma 7.1]{neves3}).

Let $c_1,$ $c_2$, and $c_3$ be three lines in $\C$ so that $c_1$  is  the real axis ($c_1^+$ being the positive part and $c_1^-$ the negative part of the real axis), $c_2,$ and  $c_3$ are the positive line segments spanned by $e^{i\theta_2}$ and $e^{i\theta_3}$ respectively, with $\pi/2<\theta_2<\theta_3<\pi.$ These curves generate three Lagrangian planes in $\R^4$ which we denote by $P_1, P_2,$ and $P_3$ respectively.

Consider a curve $\gamma(\varepsilon):[0,+\infty)\longrightarrow \C$  such that (see Figure \ref{modelo})
 \begin{itemize}
\item $\gamma(\varepsilon)$ lies in the first and second quadrant and $\gamma(\varepsilon)^{-1}(0)=0$;
\item  $\gamma(\varepsilon)\cap A(3,\infty)=c_1^+\cap  A(3,\infty)$ and $\gamma(\varepsilon)\cap A(\varepsilon,1)=(c_1^+\cup c_2\cup c_3)\cap  A(\varepsilon,1)$; 
\item $\gamma(\varepsilon)\cap B_{1}$ has two connected components $\gamma_1$ and $\gamma_2$, where $\gamma_1$ connects $c_2$ to $c_1^+$ and $\gamma_2$ coincides with $c_3$; 
\item The Lagrangian angle of $\gamma_1$, $\arg \left(\gamma_1 \frac{d\gamma_1}{ds}\right)$,  has oscillation strictly smaller than $\pi/2$.
\end{itemize}

\begin{figure}
\centering {\epsfig{file=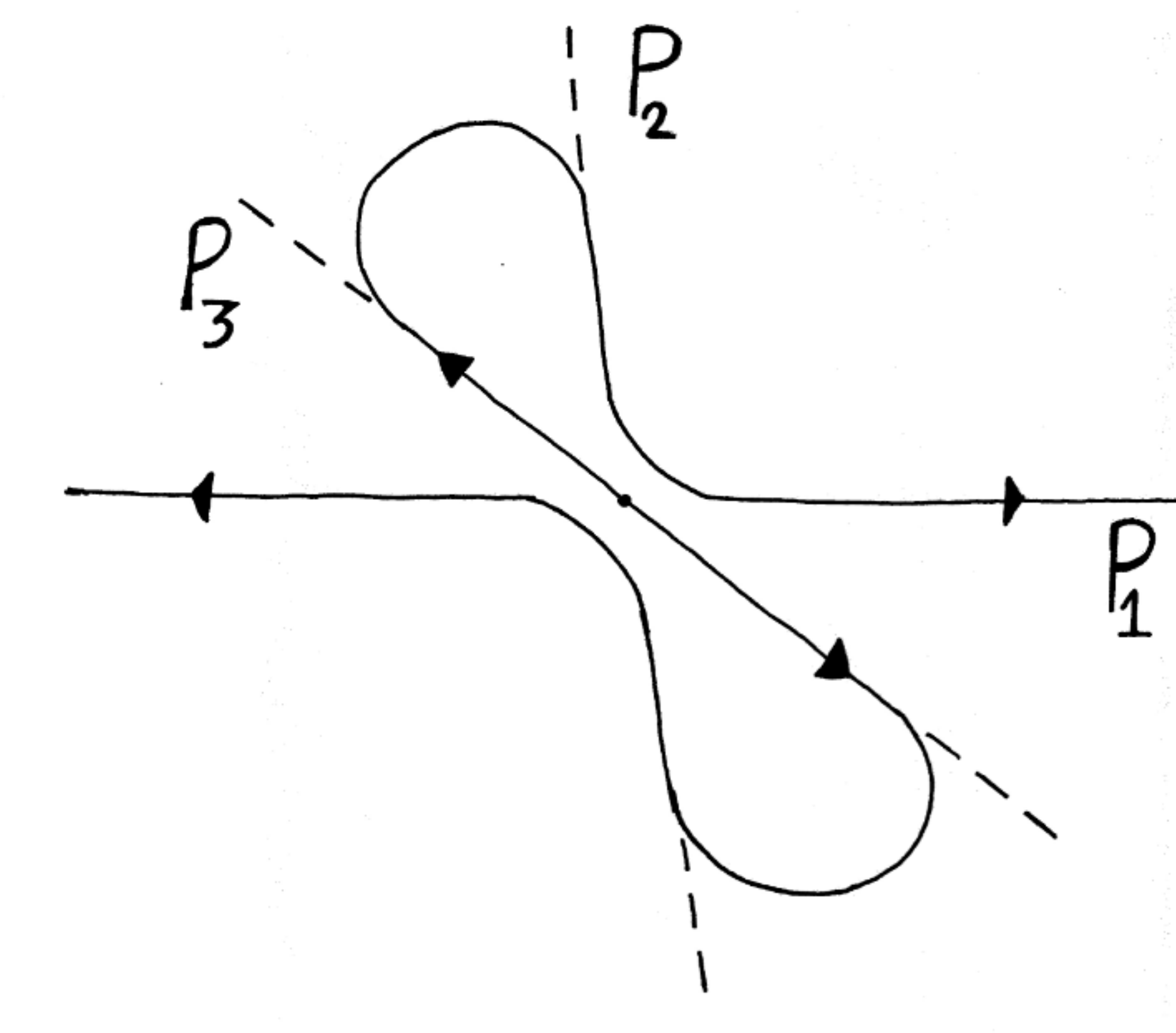, height=160pt}}\caption{Curve $\gamma(\varepsilon)\cup -\gamma(\varepsilon)$.}\label{modelo}
\end{figure}
For every $\varepsilon$ small and $R$ large we denote by $N(\varepsilon,  R)$ the Lagrangian surface corresponding to $R\gamma(\varepsilon R)$.  We remark that one  can make the oscillation for the Lagrangian angle of $\gamma_1$ as small  as desired by choosing $\theta_2$ very close to $\pi/2$.

\begin{thm}\label{singular}For all $\varepsilon$ sufficiently small and $R$ sufficiently large, there is $\delta$ so that if $L_0$ is $\delta$-close in $C^{2,\alpha}$ to $N(\varepsilon, R)$, then the Lagrangian mean curvature flow $(L_t)_{t\geq 0}$ must have a finite time singularity.
\end{thm}

\begin{rmk}\
\begin{enumerate}
\item[i)] The ideas that go into the proof of Theorem \ref{singular} are exactly the same ideas that go into the proof of \cite[Theorem A]{neves}, with the advantage of the former having less technical details.
\item [ii)] If $L_0$ is equivariant the flow reduces to an O.D.E. in which case  Theorem \ref{singular} follows from simple barrier arguments like the ones used in \cite[Section 4]{neves}.
\item[iii)] It is conceivable that a solution to mean curvature flow has a singularity at time $T$ but there are arbitrarily small $C^{2,\alpha}$ perturbations of the initial condition which are smooth up to $T+\delta$, with $\delta$ fixed. The standard example is the dumbbell degenerate neckpinch due to Angenent and Vel\'azquez.  Thus the interest of Theorem \ref{singular}.
\end{enumerate}
\end{rmk}
\begin{proof}[Sketch of proof]
Fix $\varepsilon$ small and $R$ large to be chosen later. The strategy is the following: If the flow $(L_t)_{t\geq 0}$ exists smoothly for all $t\leq 1$, there will be a singularity before some time $T=T(\varepsilon, R)$ and thus the flow cannot exist smoothly for all time.

Suppose  the theorem does not hold. We have a sequence of smooth flows $(L^i_t)_{t\geq 0}$ where $L^i_0$ converges to $N(\varepsilon, R)$. Compactness for integral Brakke motions \cite[Section 7.1]{ilmanen1} implies that, after passing to a subsequence, $(L^i_t)_{t\geq 0}$ converges to an integral Brakke motion $(M_t)_{t\geq 0}.$ Because
$$\lim_{i\to\infty}\int_{L^i_0}\mu^2\Phi(0,1)d\H^2=0,$$
we use Lemma \ref{cor-test} iii)  and conclude that $M_t$ lies inside $\mu^{-1}(0)$ for all $t$.

Let $\sigma$ denote a
smooth curve $\sigma:[0,+\infty)\longrightarrow \C$ (see Figure \ref{modelo-3}) so that $\sigma^{-1}(0)=0$, $\sigma\cup-\sigma$ is smooth at the origin, $\sigma$ has a unique self intersection, and, when restricted to $[s_0,\infty)$ for some $s_0>0$, the curve $\sigma$ can be written as the graph of a function $u$ defined over part of the negative real axis with $$\lim_{r\to-\infty}|u|_{C^{2,\alpha}((-\infty,r])}=0.$$
\begin{figure}
\centering {\epsfig{file=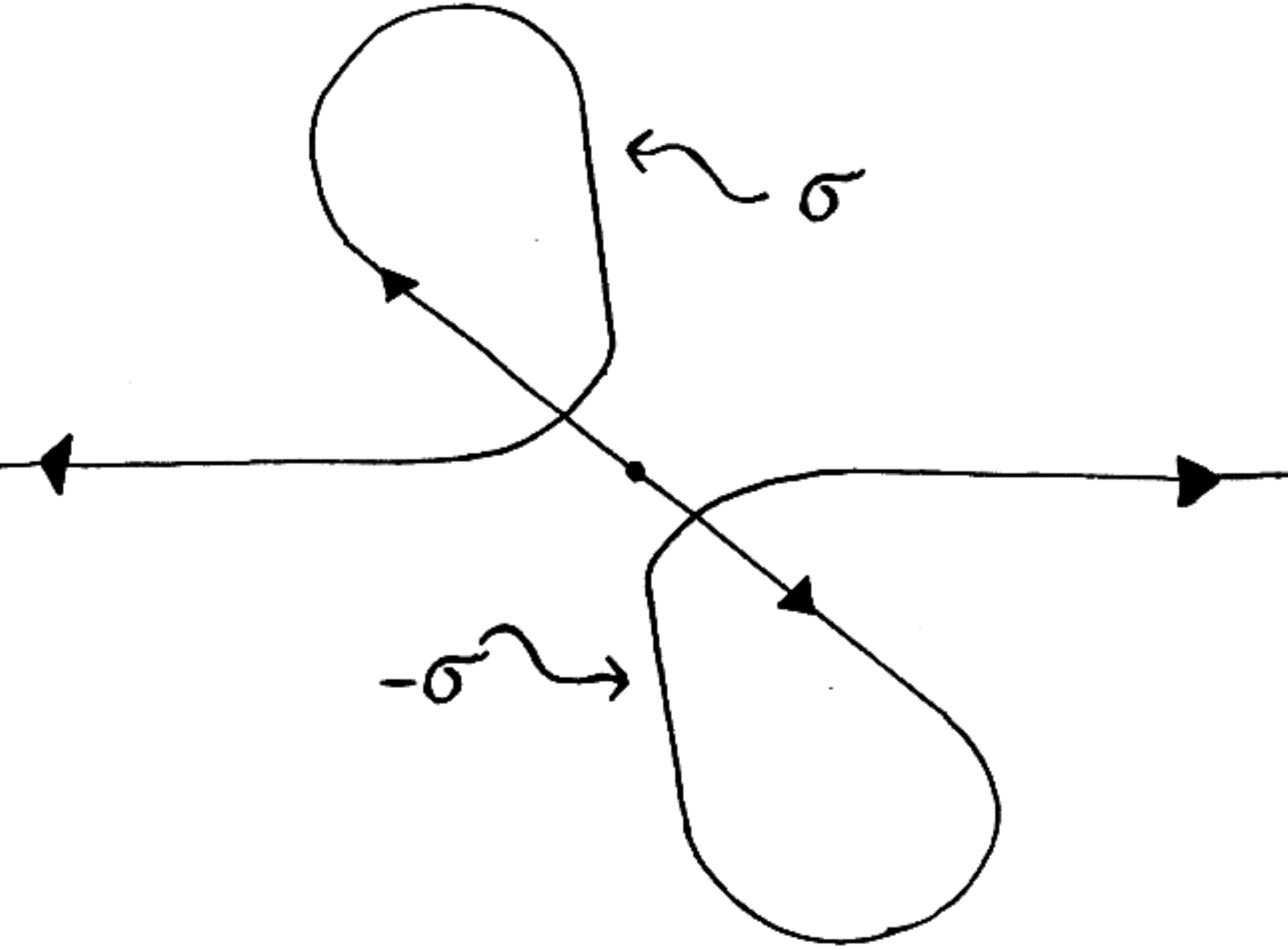, height=160pt}}\caption{Curve $\sigma\cup -\sigma$.}\label{modelo-3}
\end{figure}

\vspace{0.1in}
\noindent{\bf First Claim:}
$M_1=\{(\sigma(s)\cos \alpha, \sigma(s) \sin \alpha)\,|\, s\in [0,+\infty), \alpha\in S^1\}.$

 It is a  simple technical matter to find $R_1$ (independent of $R$ large and $\varepsilon$ small) so that, for all $i$ sufficiently large,  we have uniform $C_{loc}^{2,\alpha}$ bounds for $L^i_1$ outside $B_{R_1}$.  Moreover,  for all $R_2\geq R_1$ and $i$ sufficiently large, $L^i_0\cap B_{R_2}$ has exactly two connected components $Q^i_{1,0}$, $Q^i_{2,0}$ where $Q^i_{1,0}$, $Q^i_{2,0}$ is arbitrarily close to $(P_1+P_2)\cap B_{R_2}$, $P_3\cap B_{R_@}$ respectively. Some extra work  (see \cite[Theorem 3.1]{neves3} for details) shows that  $L^i_t\cap B_{R_2}$ can be decomposed into two connected components $Q_{1,t}^i, Q_{2,t}^i$  for all $t\leq 1$.

Fix $\nu$ small and set $S_0=2R_1$ in Theorem \ref{sex}. There is $R_2$ so that for all $R$ sufficiently large and $\varepsilon$ sufficiently small we can apply Theorem \ref{sex} to $(Q^i_{1,t})_{0\leq t\leq 1}$ (more rigorously, we should actually apply \cite[Theorem 4.1]{neves3} because $Q^i_{1,t}$ has boundary) and conclude that, for all $i$ large enough, $Q^i_{1,1}$ is $\nu$ close in $C^{2,\alpha}$ to a self-expander $Q$ in $B_{S_0}$. In  \cite[Lemma 6.1]{neves3} we showed $Q$ to be unique and so we obtain uniform $C^{2,\alpha}$ bounds  for $Q^{i}_{1,1}$ in $B_{S_0}$. Furthermore $Q_{2,0}^i$ is arbitrarily close to $P_3$ and  some technical work  (see \cite[Theorem 3.1]{neves3}) shows that $Q_{2,1}^i\cap B_{S_0}$ is also close in $C^{2,\alpha}$ to $P_3$. 

In sum, we have shown uniform $C^{2,\alpha}$ bounds for $L^i_1\cap B_{R_1}$ for all $i$ sufficiently large. As a result $M_1$ must be smooth, equivariant, and $M_1\cap B_{R_1}$ must have one connected component $\nu$-close to $Q$ and another connected component close in $C^{2,\alpha}$ to $P_3$. It is now simple to see that $M_1$ must be described by a curve $\sigma$ as claimed. 

\vspace{0.1in}
Denote by $A_1$ the area enclosed by the self-intersection of $\sigma$ and set  $T_1=2A_1/ \pi+1$. 

\vspace{0.1in}
\noindent{\bf Second  Claim:} For all $i$ sufficiently large, $L^i_t$ must have a singularity before $T_1$.

Straightforward considerations (see \cite[Theorem 5.3]{neves3}) show that while $(M_t)_{t\geq 1}$ is smooth we have the existence of  curves $\sigma_t:[0,+\infty)\longrightarrow \C$ with $\sigma_t^{-1}(0)=0,$ $\sigma_t\cup-\sigma_t$  smooth, 
$$M_t=\{(\sigma_t(s)\cos \alpha, \sigma_t(s) \sin \alpha)\,|\, s\in [0,+\infty), \alpha\in S^1\},$$
and  such that $\sigma_t$ evolves according to
$$\frac{dx}{dt}=\vec k-\frac{x^{\bot}}{|x|^2}.$$

While this flow is smooth all the curves $\sigma_t$ have a self-intersection and so we  consider $A_t$ the area enclosed by this self-intersection and  $c_t$ the boundary of the enclosed region. From Gauss-Bonnet Theorem we have
$$\int_{c_t}\langle\vec k,\nu\rangle d\H^1+\alpha_t=2\pi\implies \int_{c_t}\langle\vec k,\nu\rangle d\H^1\geq \pi,$$
where $\alpha_t \in [-\pi,\pi]$ is  the exterior angle at the non-smooth point of $c_t$, and $\nu$ the interior unit normal.
A standard formula shows that 
$$\frac{d}{dt}A_t=-\int_{c_t}\left \langle \vec k-\frac{{x^{\bot}}}{|x|^{2}},\nu\right\rangle d\H^1\leq -\pi+\int_{c_t}\left \langle\frac{{x}}{{|x|^{2}}},\nu\right\rangle d\H^1=-\pi,$$
where the last identity follows from the Divergence Theorem combined with the  fact that $c_t$ does not contain the origin in its interior.  Thus $A_t\leq A_0-\pi t$ and so a singularity must occur at time $T<T_1$.

Angenent's work \cite{angenent, angenent0} implies the singularity occurs because the loop of $\sigma_t$ collapses, i.e. $\sigma_T$ is smooth everywhere expect at a cusp point and some extra work shows the curves $\sigma_t$ become smooth and embedded for all $t>T$.  The idea for the rest of the argument is as follows and the details can be found in \cite[Theorem 5.1]{neves3}.  If $\hat \theta_t(s)$ denotes the angle that $\sigma'_t(s)$ makes with the $x$-axis we have
$$\theta_t(0)=2\hat\theta_t(0)\quad\mbox{and}\quad \theta_t(\infty)=\lim_{s\to\infty}\theta_t(s)=2\hat \theta_t(\infty).$$
Set $f(t)=\theta_t(\infty)-\theta_t(0)=2(\hat \theta_t(\infty)-\hat \theta_t(0))$. Because there is a change in the topology of $\sigma_t$ across $T$ we have from the Hopf index Theorem that $f(t)$ jumps by $2\pi$ across $T$. On the other hand, the convergence of $L^i_t$ to $M_t$ is strong around the origin and outside a large compact set, which means that $f$ must be continuous, a contradiction.
\end{proof}

\section{Open Questions}
 We now survey some open problems which could be relevant for the development of the field. One of the most important open questions is
 
 \begin{quest}\label{tapio}
 Let $L_0$ be rational, almost-calibrated, and $(L_t)_{0\leq t<T}$  a solution to Lagrangian mean curvature flow in $\C^2$ which becomes singular at time $T$. Show that $L_T$ is smooth except at finitely many points and the tangent cone at each of these points is a special Lagrangian cone with some multiplicity.
 \end{quest}

 $L_0$ is required to be rational so that Theorem \ref{thmb} can be applied and almost-calibrated because otherwise there are simple counterexamples (see \cite[Example 1.1]{neves2}). For instance, take a  noncompact curve $\sigma$ in $\C$ with a unique self intersection and  consider $L=\sigma\times\R\subset \C^2$ which is obviously Lagrangian. The solution to Lagrangian mean curvature flow  will be $L_t=\sigma_t\times\R$ which will have a whole line of singularities at some time $T$. 
 
 If one can show that for each singular point $x_0$ there is $\theta(x_0)$ so that
 \begin{equation}\label{hard}
 \lim_{t\to T}\int_{L_t}(\theta_t-\theta(x_0))\Phi(x_0,T)d\H^2=0,
 \end{equation}
 then standard arguments  prove the desired result.  We now comment on the difficulty of \eqref{hard}. For simplicity let us assume that $$\Theta_t(x,r^2)<3 \mbox{ for all $r\leq \delta$, $T-\delta^2<t<T$, and $x\in \C^2$.}$$
In this case the blow up at each singular point described in Section \ref{blowup} will be  a union of two planes $P_1$+$P_2$ by Theorem \ref{type1} and three cases can happen:  $\dim (P_1\cap P_2)=2$, $\dim (P_1\cap P_2)=0$, or  $\dim (P_1\cap P_2)=1$. In the first case the Lagrangian angles of $P_1$ and $P_2$ must be identical by the almost-calibrated condition and so we should have $\theta(x_0)=\theta(P_1)=\theta(P_2)$.  In the second case we have  from Corollary \ref{ap1} that $P_1$ and $P_2$ must have the same Lagrangian angle and so  $\theta(x_0)=\theta(P_1)=\theta(P_2)$. The third case we want to show is impossible and this is a highly non trivial matter for reasons we know explain.

In \cite{joyce}, Joyce, Lee, and Tsui found a Lagrangian $N$ with small oscillation of the Lagrangian angle and such that
$$N_t=N+t\vec v$$
solves mean curvature flow. If we blow down this flow, i.e., consider the sequence $N^i_s=\varepsilon_i N_{s/\varepsilon_i^2}$ with $\varepsilon_i$ tending to zero,  one can easily check that, for all $s<0$, $N^i_s$ converges weakly to $P_1+P_2$, where $P_1,P_2$ are Lagrangian planes with $P_1\cap P_2=\mbox{span}\{\vec v\}$.  Thus, if we rule out the third case we would be ruling out Joyce, Lee, and Tsui solutions as smooth blow ups, a clearly deep fact.

In \cite{joyce} the authors also found many examples of self-expanders and so a natural problem is

\begin{quest}\
\begin{itemize}
 \item Let $L\subset \C^n$ be a zero-Maslov class self-expander asymptotic to two planes. Show it coincides with one of the self-expanders of Joyce, Lee, and Tsui.
 \item Are there almost-calibrated translating solutions to Lagrangian mean curvature flow  besides the ones found by Joyce, Lee, and Tsui?
 \end{itemize}
\end{quest}
Castro and Lerma \cite{castro} solved the first question when $n=2$ assuming $L$ is Hamiltonian stationary as well, i..e, the Lagrangian angle is harmonic. When $n>2$ it is not known whether special Lagrangians asymptotic to two planes have to be one of the Lawlor necks, which adds interest to the proposed problem. In \cite{neves3} we  showed the blow-down $N^i_s$ of translating solutions  converges weakly to $\Sigma_s$  with
$$\Sigma_s=m_1P_1+\ldots+m_kP_k\mbox{ for all }s\leq 0\quad\mbox{and }\Sigma_s=\sqrt s \Sigma_1\mbox{ for all }s> 0,$$
where the planes $P_j$ intersect all along a line and $\Sigma_1$ is a self-expander asymptotic to $m_1P_1+\ldots+m_kP_k$. It is easy to see that $\Sigma_1=\sigma\times \mbox{line}$, where $\sigma$ is a self-expander in $\C$, and so the first step towards the second problem would be to see if one can have $k=3$ and $m_1=m_2=m_3=1$. In \cite{castro2} examples were found with unbounded Lagrangian angle.

Success in Question \ref{tapio} would make the following problem having crucial importance.

\begin{quest}
Let $L\subset \C^n$ be zero-Maslov class, almost-calibrated, and smooth everywhere except the origin where the tangent cone is a special Lagrangian cone with multiplicity. Find $(L_t)_{\varepsilon<t<\varepsilon}$ a meaningful solution to Lagrangian mean curvature flow so that $L_t$ converges to $L$ when $t$ tends to zero.
 \end{quest}

Behrndt \cite{ber} made concrete progress on this problem  when the multiplicity is one and the special Lagrangian cone is stable. If would be nice to have a solution when the tangent cone is a plane with multiplicity two.

Finally, there is no result available regarding convergence of compact Lagrangians in $\C^n$. For instance, the following question can be seen as a Lagrangian  analogue of Huisken's classical result for mean curvature flow of convex spheres \cite{huisken1}.

\begin{quest}
Find a condition on a Lagrangian torus in $\C^2$, which implies that Lagrangian mean curvature flow $(L_t)_{0<t<T}$ will become extinct at time $T$ and, after rescale, $L_t$ converges to the Clifford torus.
\end{quest}

It was shown in \cite{neves4,smo} that  $L$ can be   Hamiltonian isotopic to a Clifford Torus and the flow still develop singularities before the optimal time. As suggested by Joyce, the first natural thing would be to see what happens to small Hamiltonian perturbations of the Clifford torus.

\bibliographystyle{amsbook}

\begin{thebibliography}{99}
\bibitem{anciaux} H. Anciaux, Construction of Lagrangian self-similar solutions to the mean curvature flow in $\mathbb C^n$.  {\bf Geom. Dedicata 120} (2006), 37--48. 
\bibitem {angenent0} S. Angenent, Parabolic equations for curves on
      surfaces. I. Curves with $p$-integrable curvature.
{\bf Ann. of Math. (2) 132} (1990), 451--483. 
\bibitem {angenent} S. Angenent, Parabolic equations for curves on
      surfaces. II. Intersections,  blow-up and generalized solutions.
{\bf Ann. of Math. (2) 133} (1991), 171--215.
\bibitem{av} S. Angenent and J. Vel\'azquez, Degenerate neckpinches in mean curvature flow.
{\bf J. Reine Angew. Math. 482} (1997), 15--66. 

\bibitem{ber} T. Behrndt, Ph.D. Oxford DPhil, in preparation.
\bibitem{brakke} K. Brakke, The motion of a surface by its mean curvature. {\bf Mathematical Notes, 20}, (1978) Princeton University Press, Princeton, N.J.

\bibitem{castro} I. Castro, A.  Lerma, 
Hamiltonian stationary self-similar solutions for Lagrangian mean curvature flow in the complex Euclidean plane. 
{\bf Proc. Amer. Math. Soc. 138} (2010), 1821--1832. 
\bibitem{castro2}I. Castro, A.  Lerma, Translating solitons for Lagrangian mean curvature flow in complex Euclidean plane, preprint.
\bibitem{chen1}J. Chen and J. Li, Mean curvature flow of surface in $4$-manifolds.
{\bf Adv. Math. 163} (2001), 287--309. 
\bibitem {CL} J. Chen and J. Li, Singularity of mean curvature flow of Lagrangian submanifolds,
    {\bf Invent. Math. 156} (2004), 25--51.
    \bibitem{chen} J. Chen, J. Li, and G. Tian, Two-dimensional graphs moving by mean curvature flow.  
{\bf Acta Math. Sin.} 18 (2002), 209--224. 

\bibitem{yasha} Y. Eliashberg and L. Polterovich, Local Lagrangian $2$-knots are trivial. {\bf Ann. of Math. (2) 144} (1996), 61--76. 
\bibitem{grayson} M. Grayson, Shortening embedded curves. {\bf Ann. of Math. (2) 129} (1989), 71--111.
\bibitem{huisken1} G. Huisken, 
Flow by mean curvature of convex surfaces into spheres.
{\bf J. Differential Geom. 20} (1984), 237--266.
\bibitem {huisken} G. Huisken,
Asymptotic behavior for singularities of the mean curvature flow. {\bf J. Differential Geom. 31} (1990),
285--299.
\bibitem {ilmanen1} T. Ilmanen, Elliptic Regularization and Partial Regularity for Motion by Mean Curvature. {\bf Mem. Amer. Math. Soc. 108} (1994),  1994.
\bibitem {ilmanen} T. Ilmanen, Singularities of Mean Curvature Flow of Surfaces. Preprint.
\bibitem{joyce0} D. Joyce,   Special Lagrangian submanifolds with isolated conical singularities. II. Moduli spaces. {\bf Annals of Global Analysis and Geometry 25} (2004), 301-352. 
\bibitem{joyce} D. Joyce, Y.-I. Lee, and M.-P. Tsui, Self-similar solutions and translating solitons for Lagrangian mean curvature flow.  
{\bf J. Differential Geom. 84} (2010), 127--161. 
\bibitem {neves} A. Neves, Singularities of Lagrangian Mean Curvature Flow: Zero-Maslov class case. {\bf Invent. Math. 168} (2007), 449--484.
\bibitem {neves4} A. Neves, Singularities of Lagrangian mean curvature flow: Monotone case.  {\bf Math. Res. Lett. 17} (2010) 109--126.
\bibitem {neves2} A. Neves and G. Tian, Translating solutions to Lagrangian mean curvature flow, preprint.
\bibitem{neves3} A. Neves, Finite time singularities for Lagrangian mean curvature flow, preprint.
\bibitem{oaks} J. Oaks, Singularities and self-intersections of curves evolving on surfaces.
{\bf Indiana Univ. Math. J. 43} (1994), 959--981.
\bibitem {schoen}R. Schoen and J.  Wolfson,
Minimizing area among Lagrangian surfaces: the mapping problem. {\bf J. Differential Geom. 58} (2001), 1--86.
\bibitem{schoen1} R. Schoen and J. Wolfson, Mean curvature flow and Lagrangian embeddings, preprint.
\bibitem {Leon} L. Simon,
Lectures on geometric measure theory. {\bf Proceedings of the Centre for Mathematical Analysis, Australian
National University, 3}.

\bibitem{smo}K. Groh, M. Schwarz, K. Smoczyk, K.  Zehmisch, 
Mean curvature flow of monotone Lagrangian submanifolds. 
{\bf Math. Z. 257} (2007), 295--327.

\bibitem {smo0} K. Smoczyk,
A canonical way to deform a Lagrangian submanifold, preprint.

\bibitem{smoczyk} K. Smoczyk, Angle theorems for the Lagrangian mean curvature flow.  
{\bf Math. Z.} 240 (2002), 4, 849--883. 

\bibitem{smoczyk2} K. Smoczyk, Longtime existence of the Lagrangian mean curvature flow.  
{\bf Calc. Var. Partial Differential Equations} 20 (2004), 25--46. 






\bibitem{smoczyk-mt} K. Smoczyk and M.-T. Wang, Mean curvature flows of Lagrangians submanifolds with convex potentials.  
{\bf J. Differential Geom. 62} (2002), 243--257. 





\bibitem{thomas} R. P. Thomas and S.-T. Yau, Special Lagrangians, stable bundles and mean curvature flow. 
{\bf Comm. Anal. Geom. 10} (2002), 1075--1113. 
\bibitem {tsui} M.-P. Tsui and M.-T. Wang,
     Mean curvature flows and isotopy of maps between spheres,
     {\bf Comm. Pure Appl. Math. 57} (2004), 1110--1126.
\bibitem {Wa1} M.-T. Wang, Mean curvature flow of surfaces in Einstein four-manifolds, {\bf J. Differential Geom. 57} (2001), 301--338.

\bibitem{mt1} M.-T. Wang, Deforming area preserving diffeomorphism of surfaces by mean curvature flow.  
{\bf Math. Res. Lett. 8} (2001), 651--661. 

\bibitem{mt2} M.-T. Wang, A convergence result of the Lagrangian mean curvature flow.   Third International Congress of Chinese Mathematicians.  
{\bf AMS/IP Stud. Adv. Math. 42}  291--295,  Amer. Math. Soc., Providence, RI, 2008. 

\bibitem {mt3} M.-T. Wang,
     Long-time existence and convergence of graphic mean curvature
     flow in arbitrary codimension,
    {\bf Invent. Math. 148} (2002), 525--543.

\bibitem{wang} M.-T. Wang, Some recent developments in Lagrangian mean curvature flows. {\bf   Surveys in differential geometry. Vol. XII. Geometric flows,}  333--347, Surv. Differ. Geom., 12, Int. Press,2008. 

\bibitem{wang-2} M.-T. Wang, Lectures on mean curvature flows in higher codimensions.  {\bf Handbook of geometric analysis. No. 1,}  525Ð543,
Adv. Lect. Math. (ALM), 2008. 

\bibitem{white}B. White, A local regularity theorem for mean curvature flow. {\bf Ann. of Math.  161 }(2005), 1487--1519.

\bibitem{wolfson} J. Wolfson, Lagrangian homology classes without regular minimizers. {\bf J. Differential Geom. 71} (2005), 307--313. 


\end{thebibliography}

\vspace{20mm}

\end{document}